\documentclass[11pt]{amsart}

\usepackage{amsmath,amssymb,latexsym,soul,cite,mathrsfs}
\usepackage{color,enumitem,graphicx}
\usepackage[colorlinks=true,urlcolor=blue,
citecolor=red,linkcolor=blue,linktocpage,pdfpagelabels,
bookmarksnumbered,bookmarksopen]{hyperref}
\usepackage[english]{babel}

\usepackage[left=2.9cm,right=2.9cm,top=2.8cm,bottom=2.8cm]{geometry}
\usepackage[hyperpageref]{backref}

\usepackage[colorinlistoftodos]{todonotes}
\makeatletter
\providecommand\@dotsep{5}
\def\listtodoname{List of Todos}
\def\listoftodos{\@starttoc{tdo}\listtodoname}
\makeatother

\numberwithin{equation}{section}

\newtheorem{theorem}{Theorem}[section]
\newtheorem{proposition}[theorem]{Proposition}
\newtheorem{lemma}[theorem]{Lemma}
\newtheorem{corollary}[theorem]{Corollary}

\newtheorem{claim}[theorem]{Claim}

\newtheorem{remark}{Remark}

\begin{document}

\title[Existence of solution for a class of quasilinear problem in ......]{Existence of solution for a class of quasilinear problem in Orlicz-Sobolev space without $\Delta_2$-condition}

\author{Claudianor O. Alves}
\author{Edcarlos D. Silva }
\author{ Marcos T. O. Pimenta }

\address[Claudianor O. Alves]{\newline\indent Unidade Acad\^emica de Matem\'atica
\newline\indent
Universidade Federal de Campina Grande,
\newline\indent
58429-970, Campina Grande - PB - Brazil}
\email{\href{mailto:coalves@dme.ufcg.edu.br}{coalves@dme.ufcg.edu.br}}

\address[Edcarlos D. Silva]
{\newline\indent Intituto de Matem\'atica e Estat\'stica
\newline\indent
Universidade Federal de Goi\'as
\newline\indent
74001-970, Goi\^ania - GO, Brazil}
\email{\href{mailto:edcarlos@ufg.br}{edcarlos@ufg.br}}

\address[ Marcos T. O. Pimenta]{\newline\indent
Departamento de Matem\'atica e Computa\c{c}\~ao
\newline\indent
 Universidade Estadual Paulista (Unesp), Faculdade de Ci\^encias e Tecnologia
\newline\indent
19060-900 - Presidente Prudente - SP, Brazil }
\email{\href{mailto:pimenta@fct.unesp.br}{pimenta@fct.unesp.br}}

\pretolerance10000


\begin{abstract}
\noindent In this paper we study existence of solution for a class of problem of the type
$$
\left\{
\begin{array}{ll}
-\Delta_{\Phi}{u}=f(u), \quad \mbox{in} \quad \Omega \\
u=0, \quad \mbox{on} \quad \partial \Omega,
\end{array}
\right.
$$
where $\Omega \subset \mathbb{R}^N$, $N \geq 2$, is a smooth bounded domain, $f:\mathbb{R} \to \mathbb{R}$ is a continuous function verifying some conditions, and $\Phi:\mathbb{R} \to \mathbb{R}$ is a N-function which is not assumed to satisfy the well known $\Delta_2$-condition, then the Orlicz-Sobolev space $W^{1,\Phi}_0(\Omega)$ can be non reflexive. As main model we have the function $\Phi(t)=(e^{t^{2}}-1)/2$.  Here, we study some situations where it is possible to work with global minimization, local minimization and mountain pass theorem, however some estimates are not standard for this type of problem. 
\end{abstract}

\thanks{Claudianor Alves was partially supported by CNPq/Brazil Proc. 304036/2013-7 ; Edcarlos Silva was partially supported by
	CNPq, Brazil,  Marcos Pimenta was partially supported by FAPESP and CNPq 442520/2014-0, Brazil.  }

\maketitle

\section{Introduction}

In this paper we study existence of weak solution for a class of quasilinear problem of the type
$$
\left\{
\begin{array}{ll}
-\Delta_{\Phi}{u}=f(u), \quad \mbox{in} \quad \Omega, \\
u=0, \quad \mbox{on} \quad \partial \Omega,
\end{array}
\right.
\leqno{(P)}
$$
where $\Omega \subset \mathbb{R}^N$, $N \geq 2$, is a smooth bounded domain, $f:\mathbb{R} \to \mathbb{R}$ is a continuous function verifying some conditions which will be mentioned later on, and 
$$
\Delta_\Phi u = \mbox{div}(\phi(|\nabla u|\nabla u)
$$
where $\Phi:\mathbb{R} \to \mathbb{R}$ is a N-function of the form
$$
\Phi(t)=\int_{0}^{|t|}s\phi(s) \,ds
$$
and $\phi:[0, +\infty) \to [0, +\infty) $ is a $C^{1}$ function verifying the following conditions
$$
t \mapsto t\phi(t); \quad t>0 \;\; \quad \mbox{increasing} \quad \mbox{and} \quad t \mapsto t^2\phi(t) \quad \mbox{is convex in } \quad \mathbb{R}. \leqno{(\phi_1)}
$$
$$
\lim_{t \to 0}t\phi(t)=0, \quad  \lim_{t \to +\infty}t\phi(t)=+\infty. \leqno{(\phi_2)}
$$
$$
t \mapsto \frac{t^{2}{\phi(t)}}{\Phi(t)}, \quad \mbox{ is increasing for } \quad t>0 \quad \mbox{with} \quad
\frac{t^{2}{\phi(t)}}{\Phi(t)} \geq l >1, \quad \forall t >0. \leqno{(\phi_3)}
$$
for some $l>1$. 
$$
\frac{t^{2}\phi(t)}{\Phi(t)} \leq 1+\frac{t\phi'(t)}{\phi(t)} \leq \frac{2t^{2}\phi(t)}{\Phi(t)}, \quad \forall t >0. \leqno{(\phi_4)}
$$
If $d$ is twice the diameter of $\Omega$, then
$$
\limsup_{t \to 0^+}\frac{\Phi(t)}{\Phi(t/d)}<+\infty. \leqno{(\phi_5)}
$$
For all $A,B, q>0$ with $A/B <q$, we have
$$
\lim_{t \to +\infty}\frac{(\Phi(Bt))^{q}}{\Phi(At)}=+\infty. \leqno{(\phi_6)}
$$ 
$$
\liminf_{t \to +\infty}\frac{\Phi(t)}{t^{p}}>0, \quad \mbox{for some} \quad p > N. \leqno{(\phi_7)}
$$
The last assumption implies that the embedding
$$
W^{1,\Phi}_0(\Omega) \hookrightarrow W^{1,p}(\Omega) \quad \mbox{for some} \quad p>N
$$
is continuous. Hence, by Sobolev embedding, the  embedding
\begin{equation} \label{I0}
W^{1,\Phi}_0(\Omega)  \hookrightarrow C^{0,\alpha}(\overline{\Omega})
\end{equation}
is continuous for some $\alpha \in (0,1)$ and
\begin{equation} \label{I1}
W^{1,\Phi}_0(\Omega)  \hookrightarrow C(\overline{\Omega})
\end{equation}
is compact. The condition $(\phi_7)$ also implies that there is $C>0$ such that
$$
\|u\|_{W^{1,p}(\Omega)} \leq C\left(\int_{\Omega}\Phi(|\nabla u|)\,dx \right)^{\frac{1}{p}}, \quad \forall u \in W_0^{1,\Phi}(\Omega).
$$
From this, 
\begin{equation} \label{I1.1}
\|u\|_{C(\overline{\Omega})} \leq  C\left(\int_{\Omega}\Phi(|\nabla u|)\,dx \right)^{\frac{1}{p}}, \quad \forall u \in W_0^{1,\Phi}(\Omega).
\end{equation}
for some $C>0$.

Before continuing this section, we would like to point out that $\Phi(t)={(e^{t^2}-1)}/{2}$ and $\Phi(t)=|t|^{p}/p$ with $p>N$ satisfy $(\phi_1)-(\phi_7)$. Moreover, we would like to recall that $u \in W^{1,\Phi}_0(\Omega)$ is a weak solution of $(P)$ if
$$
\int_{\Omega}\phi(|\nabla u|)\nabla u \nabla v\,dx=\int_{\Omega}f(u)v\,dx, \quad \forall v \in W^{1,\Phi}_0(\Omega).
$$

Quasilinear elliptic problem have been considered using different assumptions on the N-function $\Phi$. Here we refer the reader to \cite{chung,JVMLED,MR1,MR2,MRep,r.R ,MugnaiPapageorgiou,fang} and references therein. In these works was considered the $\Delta_{2}$-condition which implies that
the Orlicz-Sobolev space $W^{1,\Phi}_{0}(\Omega)$ is a reflexive Banach space. This is used in order to get a nontrivial solution for elliptic problems
taking into account the weak topology.  In our work the main feature is to consider problem $(P)$ where the function $\Phi$ is not assumed to verify the $\Delta_{2}$-condition, then we cannot use that $W^{1,\Phi}_{0}(\Omega)$ is reflexive which brings serious difficulty to apply variational methods.  To overcome this difficulty, we apply the weak$^{\star}$ topology recovering the compactness required in variational methods. We would like to recall that $\Phi(t)=|t|^{p}/p$ for $p>1$ satisfies the $\Delta_2$-condition, while $\Phi(t)=(e^{t^{2}}-1)/2$ does not verify the $\Delta_2$-condition. For more details involving the $\Delta_2$-condition see Section 2.

In \cite{GKMS}, Garc\'ia-Huidobro, Khoi, Man\'asevich and K. Schmitt have studied the existence of solution for the following nonlinear eigenvalue problem
$$
\left\{
\begin{array}{ll}
-\Delta_{\Phi}{u}=\lambda \Psi (u), \quad \mbox{in} \quad \Omega \\
u=0, \quad \mbox{on} \quad \partial \Omega,
\end{array}
\right. \leqno{(P_2)}
$$
where $\Phi:\mathbb{R} \to \mathbb{R}$ is a N-function and $\Psi:\mathbb{R} \to \mathbb{R}$ is a continuous function   verifying some technical conditions. In that paper, the authors have considered the situation where the function $\Phi$ does not satisfy the well known $\Delta_2$-condition, for example, in the first part of this paper the authors considered the function
$$
\Phi(t)=e^{t^{2}}-1, \quad \forall t \in \mathbb{R}.
$$
After in  \cite{BM}, Bocea and  Mih\u{a}ilescu  made a careful study about the eigenvalues of the problem
$$
\left\{
\begin{array}{ll}
-div(e^{|\nabla u|^{2}}\nabla u)-\Delta u=\lambda u, \quad \mbox{in} \quad \Omega \\
u=0, \quad \mbox{on} \quad \partial \Omega.
\end{array}
\right. \leqno{(P_3)}
$$

Recently, da Silva, Gon\c calves and  Silva \cite{EGS} have studied the existence of multiple solutions for $(P_3)$. In their paper the $\Delta_2$-condition is not also assumed and the main tool used was the truncation of the nonlinearity and minimization of the energy functional associated to the quasilinear elliptic problem $(P)$.

The present paper was motivated by results found in \cite{BM} and \cite{GKMS} which can be applied for a class of quasilinear problems where the operator can be driven by N-function with exponential growth. Our first result uses the mountain pass theorem and we assume that $f:\mathbb{R} \to \mathbb{R}$ is a continuous function satisfying the following conditions:
$$
\lim_{t \to 0}\frac{F(t)}{\Phi(t)}=0, \leqno{(f_1)}
$$
where $F(t) = \int_0^t f(s)ds$.

There are $\theta>1, R > 0$ in such way that
$$
0<\theta F(t) \leq h(t)f(t)t, \quad |t| \geq R \leqno{(f_2)}
$$
holds true with $h(t)=\frac{\Phi(t)}{t^{2}\phi(t)}$ .

The condition $(f_2)$ suggests that $F$ is $\Phi$-superlinear, that is, the limit below holds
\begin{equation} \label{SUPER}
\lim_{|t| \to +\infty}\frac{F(t)}{\Phi(t)}=+\infty.
\end{equation}

In fact, by fixing $M>R>0$ and integrating the sentence
$$
\theta \frac{t\phi(t)}{\Phi(t)} \leq \frac{f(t)}{F(t)}\quad t \geq M >R
$$
we deduce that
\begin{equation} \label{E0}
 \frac{F(t)}{\Phi(t)} \geq \frac{F(M)}{\Phi(M)^\theta}\Phi(t)^{\theta -1} \to +\infty \quad \mbox{as} \quad t \to +\infty.
\end{equation}
A similar argument works to prove that
$$
\frac{F(t)}{\Phi(t)} \to +\infty \quad \mbox{as} \quad t \to -\infty.
$$
Here, we would like to point out that $f(t)=\frac{d}{dt}(\Phi(t))^{q}$, for $q>1$, satisfies the conditions  $(f_1)-(f_2)$, because in this case
$$
F(t)=(\Phi(t))^{q}, \quad \forall t \in \mathbb{R}.
$$

\vspace{0.5 cm}

Our first theorem is the following:

\begin{theorem} \label{T1} Assume that $(\phi_1)-(\phi_7)$ and $(f_1)-(f_2)$ hold. Then, there is $\theta^*>0$ such that if $\theta$ as in $(f_2)$ verifies $\theta > \theta^*$ the problem $(P)$ has a nontrivial solution.
\end{theorem}

To the best our knowledge the Theorem \ref{T1} is the first existence result for a class of quasilinear problem driven by a N-function with exponential growth by using the mountain pass theorem. Here, we have had serious difficulty in order to find a correct definition for the Ambrosetti-Rabinowitz condition for nonlinearity $f$, which makes the result interesting.

Our second result involves the existence of solution for a situation where the energy functional has a global minimum.  For this case, we assume the following conditions on $f$:
$$
0\leq F(t) \leq b_1(\Phi(t/d))^{s}, \quad \forall t \in \mathbb{R} \quad \mbox{and for some} \quad s \in (0,1). \leqno{(f_3)} 
$$
and
$$
F(t) \geq c_1(\Phi(t))^{\gamma}, \quad \forall t \in (0, \delta) \quad \mbox{for some} \quad \gamma \in (0,1) \quad \mbox{and} \quad \delta>0. \leqno{(f_4)}
$$

Related to $\Phi$, we assume that for any $A,B>0$
$$
\lim_{t \to 0}\frac{(\Phi(Bt))^{\gamma}}{\Phi(At)}=+\infty. \leqno{(\phi_{8})}
$$
where $\gamma$ is like in $(f_4)$.

The reader is invited to see that $\Phi(t)=(e^{t^{2}}-1)/2$ and $\Phi(t)=|t|^{p}/p$ for $p>N$ also satisfy $(\phi_{8})$.

\vspace{0.5 cm}

Our second result has the following statement

\begin{theorem} \label{T2} Assume $(f_3)-(f_4)$, $(\phi_1)-(\phi_2)$ and $(\phi_{8})$. Then, problem $(P)$ has a nontrivial solution.
\end{theorem}

Theorem \ref{T2} completes the study made in \cite{EGS} and \cite{GKMS}, in the sense that we have worked with a class of nonlinearity where the minimization arguments can be used, but it was not considered in the above references.

Our third result is associated with a concave-convex problem for the $\Phi$-Laplacian, which was introduced by Ambrosetti, Br\'ezis and Cerami \cite{ABC} for the Laplacian operator.  For this result, we suppose that $f$ is continuous with primitive $F$ of the form
$$
F(t)=\lambda \frac{(\Phi(t))^{\alpha}}{\alpha}+\frac{(\Phi(t))^{q}}{q}, \quad \forall t \in \mathbb{R}, \leqno{(f_5)}
$$
where $\lambda>0$, $\alpha \in (0,1)$ and $q>1$.

Our third result can be stated as below

\begin{theorem} \label{T3} Assume $(f_5)$ and $(\phi_1)-(\phi_{8})$. Then, problem $(P)$ has two nontrivial solutions for $\lambda$ small enough.
\end{theorem}

In the proof of Theorem \ref{T3} we will use Ekeland's Variational Principle and Mountain Pass Theorem. The Theorem \ref{T3} completes the study made in \cite{MSGC}, because in that paper the authors have considered the concave-convex case for a nonlinearity $f$ of the type
$$
f(x,t)=\lambda a(x)|t|^{\alpha-2}t+b(x)|t|^{q-2}t, \quad \forall t \in \mathbb{R},
$$
where $a(x),b(x),\alpha$ and $q$ satisfy some technical conditions.

Before concluding this introduction we would like to point out that in the references above mentioned it was showed that if $0 \in \partial I(u)$ and $u$ is a minimum point of $I$, then $u$ is a weak solution of the problem, where $I$ denotes the functional energy associate with the problem and $\partial I(u)$ denotes the subdifferential of $I$ at $u$ . Here, after a careful study we have improved this information, in the sense that  we have proved that if $u$ is a critical point of $I$, which means $0 \in \partial I(u)$, then $u$ is a weak solution for problem. In our opinion this is a very important information for this class of problem, for more details see Proposition \ref{Lema2} and Corollary \ref{pontocritico} in Section 3. 

\section{Basics on Orlicz-Sobolev spaces}

In this section we recall some properties of Orlicz and Orlicz-Sobolev spaces, which can be found in \cite{Adams, RR}. First of all, we recall that a continuous function $\Phi : \mathbb{R} \rightarrow [0,+\infty)$ is a
N-function if:
\begin{description}
	\item[$(i)$] $\Phi$ is convex.
	\item[$(ii)$] $\Phi(t) = 0 \Leftrightarrow t = 0 $.
	\item[$(iii)$] $\displaystyle\lim_{t\rightarrow0}\frac{\Phi(t)}{t}=0$ and $\displaystyle\lim_{t\rightarrow+\infty}\frac{\Phi(t)}{t}= +\infty$ .
	\item[$(iv)$] $\Phi$ is even.
\end{description}
We say that a N-function $\Phi$ verifies the $\Delta_{2}$-condition, if 
\[
\Phi(2t) \leq K\Phi(t),\quad \forall t\geq t_0,
\]
for some constants $K,t_0 > 0$. In what follows, fixed an open set $\Omega \subset \mathbb{R}^{N}$ and a N-function $\Phi$, we define the Orlicz space associated with $\Phi$ as follows
\[
L^{\Phi}(\Omega) = \left\{  u \in L_{loc}^{1}(\Omega) \colon \ \int_{\Omega} \Phi\Big(\frac{|u|}{\lambda}\Big)dx < + \infty \ \ \mbox{for some}\ \ \lambda >0 \right\}.
\]
The space $L^{\Phi}(\Omega)$ is a Banach space endowed with the Luxemburg norm given by
\[
\Vert u \Vert_{\Phi} = \inf\left\{  \lambda > 0 : \int_{\Omega}\Phi\Big(\frac{|u|}{\lambda}\Big)dx \leq1\right\}.
\]
The complementary function ${\Phi}^*$ associated with $\Phi$ is given
by the Legendre's transformation, that is,
\[
{\Phi}^*(s) = \max_{t\geq 0}\{ st - \Phi(t)\}, \quad  \mbox{for} \quad s\geq0.
\]
The functions $\Phi$ and ${\Phi}^*$ are complementary each other. Moreover, we also have a Young type inequality given by
\[
st \leq \Phi(t) + {\Phi}^*(s), \quad \forall s, t\geq0.
\]
Using the above inequality, it is possible to prove a H\"older type inequality, that is,
\[
\Big| \int_{\Omega}uvdx \Big| \leq 2 \Vert u \Vert_{\Phi}\Vert v \Vert_{ \Phi^*},\quad \forall u \in L^{\Phi}(\Omega) \quad \mbox{and} \quad \forall v \in L^{\Phi^*}(\Omega).
\]
The corresponding Orlicz-Sobolev space is defined as follows
\[
W^{1, \Phi}(\Omega) = \Big\{ u \in L^{\Phi}(\Omega) \ :\ \frac{\partial u}{\partial x_{i}} \in L^{\Phi}(\Omega), \quad i = 1, ..., N\Big\},
\]
endowed with the norm
\[
\Vert u \Vert_{1,\Phi} = \Vert \nabla u \Vert_{\Phi} + \Vert u \Vert_{\Phi}.
\]

The space $W_0^{1,\Phi}(\Omega)$ is defined as the closure of $C_0^{\infty}(\Omega)$ with respect to Orlicz-Sobolev norm above.

The spaces $L^{\Phi}(\Omega)$, $W^{1, \Phi}(\Omega)$ and $W_0^{1, \Phi}(\Omega)$ are separable and reflexive, when $\Phi$ and ${\Phi}^*$ satisfy the $\Delta_{2}$- condition.

If $E^{\Phi}(\Omega)$ denotes the closure of $L^{\infty}(\Omega)$ in $L^{\Phi}(\Omega)$ with respect to the norm $\|\,\,\|_{\Phi}$, then  $L^{\Phi}(\Omega)$ is the dual space of $E^{\Phi^*}(\Omega)$, while $L^{\Phi^*}(\Omega)$ is the dual space of $E^{\Phi}(\Omega)$. Moreover, $E^{\Phi}(\Omega)$ and $E^{\Phi^*}(\Omega)$ are separable spaces and any continuous linear functional $M:E^{\Phi}(\Omega) \to \mathbb{R}$ is of the form
$$
M(v)=\int_{\Omega}v(x)g(x)\,dx \quad \mbox{for some} \quad g \in L^{\Phi^*}(\Omega).
$$
When $\Phi$ verifies the $\Delta_2$-condition, we have that $E^{\Phi}(\Omega)=L^{\Phi}(\Omega)$.

Before concluding this section, we would like to state a lemma whose proof follows directly of a result found in  Donaldson \cite[Proposition 1.1]{donaldson}.

 \begin{lemma}  \label{Estrela} Assume that $\Phi$ is a N-function and $\Phi^*$ verifies the $\Delta_2$-condition. If $(u_n) \subset W^{1,\Phi}_0(\Omega) $ is a bounded sequence, then there are a subsequence of $(u_n)$, still denoted by itself, and $u \in  W^{1,\Phi}_0(\Omega)$ such that
 $$
 u_n \stackrel{*}{\rightharpoonup} u \quad \mbox{in} \quad W^{1,\Phi}_0(\Omega)
 $$
 and
 $$
 \int_{\Omega}u_n v \,dx  \to  \int_{\Omega}u v \,dx, \quad  \int_{\Omega}\frac{\partial u_n}{\partial x_i} w \,dx  \to  \int_{\Omega}\frac{\partial u}{\partial x_i} w \,dx, \quad \forall v,w \in  E^{\Phi^*}(\Omega)=L^{\Phi^*}(\Omega).
 $$
 \end{lemma}	

The above lemma is crucial when we are working in a situation where the space $W^{1,\Phi}_0(\Omega)$ is not reflexive, for example if $\Phi(t)=(e^{t^2}-1)/2$. However, if $\Phi(t)=|t|^{p}/p$ and $p>1$, the above lemma is not necessary  because $\Phi$ satisfies the $\Delta_2$-condition, and so, $W^{1,\Phi}_0(\Omega)$ is reflexive. Here we would like to point out that the condition $(\phi_3)$ ensures that $\Phi^*$ verifies the $\Delta_2$-condition, for more details see Fukagai, Ito and Narukawa \cite{FN}. From this, we can apply the above lemma in the present paper.

\section{Mountain pass }

The main goal of this section is proving Theorem \ref{T1}, then throughout this section we assume the assumptions of this theorem. We start by recalling that the  conditions $(\phi_1)-(\phi_4)$ do not imply that  $\Phi$ satisfies the $\Delta_2$-condition, then $W_0^{1,\Phi}(\Omega)$ can be non reflexive. In the case where we lose the $\Delta_2$-condition, it is well known that there is $u \in W_0^{1,\Phi}(\Omega)$ such that 
$$
\int_{\Omega}\Phi(|\nabla u|)\,dx=+ \infty.
$$ 
However, independent of $\Delta_2$-condition, the condition $(f_1)$  always guarantees that 
$$
\left|\int_{\Omega}F(u)\,dx\right|<+\infty, \quad \forall u \in W_0^{1,\Phi}(\Omega).
$$
Having this in mind, the energy functional $I:W_0^{1,\Phi}(\Omega) \to \mathbb{R}\cup \{+\infty\}$ associated with $(P)$  given by
$$
I(u)=\int_{\Omega }\Phi(|\nabla u|)dx -\int_{\Omega}F(u)dx,
$$
is well defined.  Hereafter, we denote by $D(I) \subset W_0^{1,\Phi}(\Omega)$ the  set 
$$
D(I)=\left\{u \in W_0^{1,\Phi}(\Omega)\,:\, \int_{\Omega}\Phi(|\nabla u)\,dx<+\infty\right\}.
$$
The reader must observe that $D(I)=W_0^{1,\Phi}(\Omega)$ when $\Phi$ satisfies the  $\Delta_2$-condition.

As an immediate consequence of the above remarks, we cannot guarantee that $I$ belongs to $C^{1}(W_0^{1,\Phi}(\Omega),\mathbb{R})$. However, the functional $J:W_0^{1,\Phi}(\Omega) \to \mathbb{R} $ given by
$$
J(u)=\int_{\Omega}F(u)dx
$$
belongs to $C^{1}(W_0^{1,\Phi}(\Omega),\mathbb{R})$ with
$$
J'(u)v=\int_{\Omega}f(u)v \,dx, \quad \forall u,v \in W_0^{1,\Phi}(\Omega).
$$
This can be done using Lebesgue Convergence Theorem and the fact that $f$ is a continuous function. Related to the functional $Q:W_0^{1,\Phi}(\Omega) \to \mathbb{R} \cup\{+\infty\}$ given by
\begin{equation} \label{Q}
Q(u)=\int_{\Omega}\Phi(|\nabla u|)\,dx
\end{equation}
we know that it is continuous, strictly convex   and  l.s.c. with respect to the weak$^*$ topology. Moreover, $Q \in C^{1}(W_0^{1,\Phi}(\Omega),\mathbb{R})$ when $\Phi$ satisfies the  $\Delta_2$-condition.

From the above commentaries, in the present paper we will use a minimax method developed by Szulkin  \cite{Szulkin}. In this sense, we will say that $u \in D(I)$ is a critical point for $I$ if $0 \in \partial I(u)$, 
where 
$$
\partial I(u)=\left\{\chi \in (W_0^{1,\Phi}(\Omega))'\,:\, Q(v) - Q(u) - J'(u)(v-u)\,dx \geq \chi(v-u), \,\ \forall v \in W_0^{1,\Phi}(\Omega) \right\}
$$ 
We recall that $\partial I(u)$ is the subdifferential of $I$ at $u$. Thereby, $u \in D(I)$ is a critical point for $I$ if 
$$
Q(v)-Q(u) \geq J'(u)(v-u), \quad \forall v \in W_0^{1,\Phi}(\Omega),
$$
or equivalently
\begin{equation} \label{E1}
\int_{\Omega}\Phi(|\nabla v|)\,dx - \int_{\Omega}\Phi(|\nabla u|)\,dx \geq \int_{\Omega}f(u)(v-u)\,dx, \quad \forall v \in W_0^{1,\Phi}(\Omega).
\end{equation}

If $\Phi$ satisfies the $\Delta_2$-condition, the functional $I \in C^{1}(W_0^{1,\Phi}(\Omega),\mathbb{R})$ and the last inequality is equivalent to 
\begin{equation} \label{E1.0}
I'(u)v=0, \quad \forall v \in W_0^{1,\Phi}(\Omega),
\end{equation}
or yet
$$
\int_{\Omega}\phi(|\nabla u|)\nabla u \nabla v\,dx=\int_{\Omega}f(u)v\,dx, \quad \forall v \in W^{1,\Phi}_0(\Omega),
$$
showing that $u$ is a weak solution of $(P)$. However, when $\Phi$ does not satisfy the $\Delta_2$-condition the above conclusion is not immediate and a careful analysis must be done, for more details see Lemma \ref{Lema0} below.

Hereafter, we denote by $\|\,\,\, \|$ the usual norm in  $W_0^{1,\Phi}(\Omega)$ given by
$$
\|u\|=\inf\left\{\lambda >0\,:\, \int_{\Omega}\Phi\left(\frac{|\nabla u|}{\lambda}\right)\,dx \leq 1\right\}.
$$
Since $\Phi$ is not assumed to satisfy the $\Delta_2$-condition, we cannot claim that $\|\,\,\,\|$ is an equivalent norm to induced norm  by $W^{1,\Phi}(\Omega)$. However, it is very important to point out that we have a Poincar\'e type inequality which can be stated of the form
\begin{equation} \label{PI}
\int_{\Omega} \Phi({|u|}/{d}) \leq \int_{\Omega}\Phi(|\nabla u|)\, dx \quad \forall u \in W_0^{1,\Phi}(\Omega),
\end{equation}
where $d=2 \, diam(\Omega)$. For more details see \cite[Lemma 2.1]{GKMS}.

Hereafter, we will denote by $\text{dom}(\phi(t)t) \subset W^{1,\Phi}_0(\Omega)$ the following set
$$
\text{dom}(\phi(t)t)=\left\{ u \in W^{1,\Phi}_0(\Omega)\,:\, \int_{\Omega}\Phi^*(\phi(|\nabla u|)|\nabla u|)\,dx<\infty \right\}.
$$
As $\Phi^*$ verifies $\Delta_2$-condition, the above set can be written of the form
$$
\text{dom}(\phi(t)t)=\left\{ u \in W^{1,\Phi}_0(\Omega)\,:\, \phi(|\nabla u|)|\nabla u|) \in L^{\Phi^*}(\Omega) \right\}.
$$
The set $\text{dom}(\phi(t)t)$ is not empty, because it is easy to see that $C_0^{\infty}(\Omega) \subset \text{dom}(\phi(t)t)$ .

\begin{lemma} \label{dominio} For each $u \in D(I)$, there is a sequence $(u_n) \subset  \text{dom}(\phi(t)t)$ such that 
$$
\int_{\Omega}\Phi(|\nabla u_n|)\,dx \leq \int_{\Omega}\Phi(|\nabla u|)\,dx  \quad \mbox{and} \quad \|u-u_n\| \leq {1}/{n}.
$$	
\end{lemma}

\begin{proof} For each $\epsilon \in (0,1] $, we know by a result found in  \cite[Lemma 4.1]{MT} that 
$v_\epsilon=(1-\epsilon) u \in \text{dom}(\phi(t)t)$.	By convexity of $\Phi$, it follows that
$$
\int_{\Omega}\Phi(|\nabla v_\epsilon|)\,dx \leq \int_{\Omega}\Phi(|\nabla u|)\,dx, \quad \forall \epsilon \in (0,1].
$$
On the other hand, we claim that 
\begin{equation} \label{D1}
v_\epsilon \to u \quad \mbox{in} \quad W^{1,\phi}_0(\Omega) \quad \mbox{as} \quad \epsilon \to 0.
\end{equation}
Indeed, fixed $\delta >0$, for all $\epsilon \in (0,\delta)$  we have 
$$
\frac{\Phi(|\nabla u - \nabla v_\epsilon|)}{\delta}= \frac{\Phi(|\epsilon\nabla u|)}{\delta} \leq \frac{\epsilon}{\delta}\Phi(|\nabla u|) \leq \Phi(|\nabla u|) \in L^{1}(\Omega).
$$
Applying the Lebesgue's Theorem, we get
$$
\int_{\Omega}\frac{\Phi(|\nabla u - \nabla v_\epsilon|)}{\delta}\,dx \to 0 \quad \mbox{as} \quad \epsilon \to 0.
$$
Then
$$
\|u-v_\epsilon \| < \delta 
$$
for $\epsilon$ small enough, showing the desired result.  	
\end{proof}

Our next lemma establishes that a critical point $u$ in the sense (\ref{E1}) is a weak solution for $(P)$ if $u \in \text{dom}(\phi(t)t)$.

\begin{lemma} \label{Lema0} Let $u \in  D(I) $ be a critical point of $I$. If $u \in \text{dom}(\phi(t)t)$, then it is a weak solution for $(P)$, that is,
$$
\int_{\Omega}\phi(|\nabla u|)\nabla u \nabla w \, dx=\int_{\Omega}f(u)w \, dx, \quad \forall w \in W_0^{1,\Phi}(\Omega).
$$	
\end{lemma}
\begin{proof} \, By following the arguments found in Garc\'ia-Huidobro, Khoi, Man\'asevich and Schmitt \cite{GKMS}, the directional derivative $\frac{\partial Q(u)}{\partial v}$ given by
$$
\frac{\partial Q(u)}{\partial v}=\lim_{t \to 0}\frac{Q(u+tv)-Q(u)}{t}
$$
exists for all $v \in D(I) \cap \text{dom}(\phi(t)t)$ with
$$
\frac{\partial Q(u)}{\partial v}= \int_{\Omega}\phi(|\nabla u|)\nabla u \nabla v \, dx. 
$$
Since $J \in C^{1}(W_0^{1,\Phi}(\Omega),\mathbb{R})$, we must have
$$
\frac{\partial J(u)}{\partial v}= \int_{\Omega}f(u)v \, dx, \quad \forall v \in D(I) \cap \text{dom}(\phi(t)t).
$$
From this,
$$
\frac{\partial I(u)}{\partial v}=\frac{\partial Q(u)}{\partial v}-\frac{\partial J(u)}{\partial v}, \quad \forall v \in D(I) \cap \text{dom}(\phi(t)t)
$$
and so,
$$
\frac{\partial I(u)}{\partial v}=\int_{\Omega}\phi(|\nabla u|)\nabla u \nabla v \, dx- \int_{\Omega}f(u)v \, dx, \quad \forall u,v \in D(I) \cap \text{dom}(\phi(t)t).
$$
On the other hand, by (\ref{E1}),
$$
\int_{\Omega}\Phi(|\nabla u+tv|)\,dx - \int_{\Omega}\Phi(|\nabla u|)\,dx \geq t\int_{\Omega}f(u)v\,dx, \quad \forall v \in D(I) \cap \text{dom}(\phi(t)t)\quad \mbox{and} \quad t \in \mathbb{R},
$$
which leads to 
$$
\frac{\partial Q(u)}{\partial v}=\lim_{t \to 0^{+}}\frac{\displaystyle \int_{\Omega}\Phi(|\nabla u+tv|)\,dx - \int_{\Omega}\Phi(|\nabla u|)\,dx}{t}\geq \int_{\Omega}f(u)v\,dx=\frac{\partial J(u)}{\partial v}, 
$$
or equivalently,
$$
\frac{\partial I(u)}{\partial v} \geq 0, \quad \forall v \in D(I) \cap \text{dom}(\phi(t)t).
$$
Since $v$ is arbitrary and $-v \in D(I) \cap \text{dom}(\phi(t)t)$, the last inequality gives
$$
\frac{\partial I(u)}{\partial v}=0, \quad \forall v \in D(I) \cap \text{dom}(\phi(t)t),
$$
and so,
$$
\int_{\Omega}\phi(|\nabla u|)\nabla u \nabla v \, dx = \int_{\Omega}f(u)v\,dx, \quad \forall v \in D(I) \cap \text{dom}(\phi(t)t). 
$$
In particular,
$$
\int_{\Omega}\phi(|\nabla u|)\nabla u \nabla v \, dx = \int_{\Omega}f(u)v\,dx, \quad \forall v \in C_0^{\infty}(\Omega). 
$$
Now  the result follows by using the density of $C_0^{\infty}(\Omega)$ in $W_0^{1,\Phi}(\Omega)$ together with the fact that $\phi(|\nabla u|)|\nabla u| \in L^{\Phi^{*}}(\Omega)$. 
\end{proof}

The next result shows that $I$ possesses the mountain pass geometry.

\begin{lemma} \label{Lema1} The functional $I$ satisfies the mountain pass geometry, that is, \\
\noindent $(a)$ \, There are $r,\rho >0$ such that
$$
I(u) \geq \rho \quad \mbox{for} \quad \int_{\Omega}\Phi(|\nabla u|)\,dx=r.
$$
\noindent $(b)$ \, There is $e \in W_0^{1,\Phi}(\Omega)$ with $\displaystyle \int_{\Omega}\Phi(|\nabla e|)\,dx>r$ and $I(e)<0$.
\end{lemma}
\begin{proof} We begin recalling that by $(f_1)$, given $\epsilon >0$ there is $r>0$ such that
$$
F(t) \leq \epsilon \Phi(t) \quad \mbox{for} \quad |t| \leq r.
$$	
Combining the last inequality with (\ref{I1}), it follows that
$$
I(u) \geq \int_{\Omega}\Phi(|\nabla u|)\,dx - \epsilon \int_{\Omega}\Phi(u)\, dx, \quad \mbox{for} \quad \int_{\Omega}\Phi(|\nabla u|)\,dx=r.
$$
Now, by $(\phi_5)$ and \eqref{I1.1} there exists $C>0$ such that
$$
\Phi(u) \leq C\Phi(d^{-1} \, u) \quad \mbox{for} \quad \int_{\Omega}\Phi(|\nabla u|)\,dx =r.
$$
Using the last inequality together with Poincar\'e inequality (\ref{PI}), we get
$$
I(u) \geq (1-\epsilon C)\int_{\Omega}\Phi(|\nabla u|)\,dx=\rho , \quad \mbox{for} \quad \int_{\Omega}\Phi( |\nabla u|)\,dx=r \quad \mbox{and} \quad \rho=(1-\epsilon C)r>0,
$$
showing $(a)$.  Now, we will prove $(b)$. To this end, we set  
$\Psi \in C^{\infty}(\overline{\Omega}) \cap W^{1,\Phi}_0(\Omega) $ with
$$	
\Psi(x)  >0 \quad \forall x \in \Omega, \quad \Psi(x)=0 \quad \forall x\in \partial \Omega,
$$ 
and 
$$
\quad A=|\nabla \Psi|_{\infty, \Omega} \quad \mbox{and} \quad  B=\inf_{x \in \Omega_0 }\Psi(x) \quad (\Omega_0 \subset \subset \Omega ).
$$
By (\ref{E0}), there are $A_0,B_0 > 0$ such that
$$
F(t) \geq A_0 \Phi(t)^{\theta}-B_0, \quad \forall t \in \mathbb{R}.
$$
Hence, for any $t>0$, we mention that
$$
\begin{array}{ll}
I(t \Psi) & \leq \displaystyle \int_{\Omega}\Phi(t |\nabla \Psi|)\,dx - A_0\int_{\Omega}\Phi(t \Psi)^{\theta} \, dx + B_0|\Omega|, \\
\mbox{} & \mbox{}\\
& \leq \displaystyle \int_{\Omega}\Phi(t |\nabla \Psi|_\infty)\,dx - A_0\int_{\Omega_0}\Phi(t \Psi)^{\theta} \, dx + B_0|\Omega|,  \\
\mbox{} & \mbox{}\\
& \leq  C_1\Phi(At)-C_2(\Phi(Bt))^{\theta}+B_0|\Omega|.
\end{array}
$$
Now, fixing $\theta^*>0$ such that $\frac{A}{B}<\theta^*$ and $\theta > \theta^*$, the condition $(\phi_6)$ leads to  
$$
I(t\Psi) \to -\infty \quad \mbox{as} \quad t \to +\infty,
$$
showing $(b)$. 	

\end{proof}

\begin{remark} In the proof of the last lemma we have used the condition $(\phi_6)$, but the reader is invited to observe that it is not necessary when $\Omega$ contains a ball $B_r(x_0)$ with $r>1$, because in this case it is easy to build a function $\Psi \in C^{\infty}(\overline{\Omega}) \cap W^{1,\Phi}_0(\Omega) $ verifying
$$	
\Psi(x)  >0 \quad \forall x \in \Omega, \quad \Psi(x)=0 \quad \forall x\in \partial \Omega \quad \mbox{and} \quad 
\quad A=|\nabla \Psi|_{\infty, \Omega} < B=\inf_{x \in \Omega_0 }\Psi(x) \quad (\Omega_0 \subset \subset \Omega ).
$$ 
Using this information together with the fact that $\Phi$ is increasing for $t\geq 0$, we get
$$
I(t \Psi) \leq C_1\Phi(Bt)-C_2(\Phi(Bt))^{\theta}+B_0|\Omega| \to -\infty \quad \mbox{as} \quad t \to +\infty.
$$
\end{remark}

The next result establishes that any $(PS)$ sequence of $I$ is bounded.  We recall that $(u_n) \subset W_0^{1}(\Omega)$ is a $(PS)$ sequence at level $c \in \mathbb{R}$, if there is $\tau_n \to 0 $ such that
\begin{equation} \label{sequencia1}
I(u_n) \to c \quad \mbox{as} \quad n \to +\infty
\end{equation}
and
\begin{equation} \label{sequencia2}
\int_{\Omega}\Phi(|\nabla v|)\,dx-\int_{\Omega}\Phi(|\nabla u_n|)\,dx \geq \int_{\Omega}f(u_n)(v-u_n)\,dx- \tau_n\|v-u_n\|, \quad \forall v \in W^{1,\Phi}_0(\Omega) \quad \mbox{and} \quad n \in\mathbb{N}.
\end{equation}

In the sequel we say that $I$ satisfies the $(PS)$ condition, if any $(PS)$ sequence possesses a convergent subsequence in $W^{1,\Phi}_0(\Omega)$ in the strong topology. However, we would like point out that by (\ref{sequencia1}), if $(u_n)$ is a $(PS)$ sequence for $I$, then $(u_n) \subset D(I)$.

\begin{proposition}  ({\bf Main Proposition}) \label{Lema2} If $(u_n) \subset W_0^{1,\Phi}(\Omega)$ is a $(PS)$ sequence for $I$, then $(u_n)$ is bounded and there exists $u \in  D(I) \cap \text{dom}(\phi(t)t) $ such that for some subsequence, still denoted by itself, we have
$$
\int_{\Omega}f(u_n)v\, dx  \to \int_{\Omega}f(u_n)v\,dx \quad \forall v \in W_0^{1,\Phi}(\Omega),
$$
$$
\int_{\Omega}F(u_n)\, dx  \to \int_{\Omega}F(u)\,dx
$$
and
$$
\int_{\Omega}\Phi(|\nabla u_n |)\,dx \to \int_{\Omega}\Phi(|\nabla u |)\,dx.
$$
As a byproduct of the above limits, we derive that $u$ is a critical point of $I$ and 
$$
I(u_n) \to I(u). 
$$
\end{proposition}
\begin{proof} Our first step is showing that any (PS) sequence $(u_n)$ is bounded. To this end, consider the sequence
$$
v_n(x)=\frac{\Phi(u_n)(x)}{u_n(x)\phi(u_n(x))}, \quad x \in \Omega.
$$	
A direct computation leads to
$$
\nabla v_n=\left[1-\frac{\Phi(u_n)}{u_n^{2}\phi(u_n)}\left[1+\frac{u_n\phi'(u_n)}{\phi(u_n)}\right]  \right] \nabla u_n,
$$
then by $(\phi_4)$,
\begin{equation} \label{E2}
|\nabla v_n| \leq |\nabla u_n| \quad \forall n \in \mathbb{N}.
\end{equation}
On the other hand, $(\phi_3)$ also gives
\begin{equation} \label{E3}
|v_n(x)| \leq \frac{1}{l} |u_n(x)| \quad \forall x \in \Omega.
\end{equation}
From (\ref{E2})-(\ref{E3}), $v_n \in D(I)$ with
$$
\int_{\Omega}\Phi(|\nabla v_n|)\,dx \leq \int_{\Omega}\Phi(|\nabla u_n|)\,dx, \quad \forall n \in \mathbb{N}.
$$
Applying (\ref{sequencia2}) with $v=u_n+tv_n$  and taking the limit as $t \to 0^+$ we get
$$
\int_{\Omega}\phi(|\nabla u_n|)\nabla u_n \nabla v_n \geq  \int_{\Omega}f(u_n)v_n- \tau_n\|v_n\| \geq \int_{\Omega}f(u_n)v_n- |\tau_n|\|u_n\|, \quad \forall n \in\mathbb{N},
$$	
that is,
$$
\frac{\partial I(u_n)}{\partial v_n} \geq - |\tau_n|\|u_n\|, \quad \forall n \in\mathbb{N}.
$$
Combining the above informations, we obtain
$$
c+1 \geq I(u_n)-\frac{1}{\theta}\frac{\partial I(u_n)}{\partial v_n }-\frac{1}{\theta}|\tau_n|\|u_n\|, \quad \forall n \in\mathbb{N},
$$
from where it follows that
$$
c +1 \geq \displaystyle  \int_{\Omega}\Phi(|\nabla u_n|)\,dx - \frac{1}{\theta}\int_{\Omega}\phi(|\nabla u_n|)|\nabla u_n|^{2}S(u_n)\,dx + \displaystyle \frac{1}{\theta}\int_{\Omega}(f(u_n)u_nh(u_n)-\theta F(u_n))\,dx -\frac{1}{\theta}|\tau_n|\|u_n\|,
$$
where
$$
h(t)=\frac{\Phi(t)}{t^{2}\phi(t)} \quad \mbox{and} \quad S(t)=1-\frac{\Phi(t)}{t^{2}\phi(t)}\left[1+\frac{t \phi'(t)}{\phi(t)} \right].
$$
From  $(f_2)$ and $(\phi_4)$, $S(t) \leq 0$ for all $t \in \mathbb{R}$, and so
$$
c+1 \geq \int_{\Omega}\Phi(|\nabla u_n|)\,dx -K-\frac{1}{\theta}|\tau_n|\|u_n\|, \forall n \in \mathbb{N},
$$
for some $K>0$. Supposing by contradiction that $(u_n)$ possesses a subsequence, still denoted by itself, satisfying
$$
\|u_n\| \to +\infty \quad \mbox{as} \quad n \to +\infty,
$$
we must have for $n$ large enough
$$
\int_{\Omega}\Phi(|\nabla u_n|)\,dx \geq \|u_n\|.
$$
Hence, for $n$ large enough
$$
c+1 \geq \left( 1-\frac{1}{\theta}|\tau_n|\right)\|u_n\|-K \to +\infty \quad \mbox{as} \quad n \to +\infty,
$$
which is a contradiction. The above analysis shows that $(u_n)$ is a bounded sequence in $W^{1,\Phi}_0(\Omega)$. Now, we will show that $(u_n)$ has a subsequence strongly convergent in $W^{1,\Phi}_0(\Omega)$. In order to do that, taking into account (\ref{I1}), there exists $u \in D(I) \cap C(\overline{\Omega})$ and a subsequence of $(u_n)$, still denoted by itself, such that
$$
u_n \to u \quad \mbox{in} \quad C(\overline{\Omega}).
$$
The last limit permits to conclude that 
$$
\int_{\Omega}f(u_n)v\, dx  \to \int_{\Omega}f(u_n)v\,dx, \quad \forall v \in W_0^{1,\Phi}(\Omega),
$$
and
\begin{equation} \label{FUNT1}
\int_{\Omega}F(u_n)\, dx  \to \int_{\Omega}F(u)\,dx.
\end{equation}
Since $(I(u_n))$ is bounded, we will suppose that for some subsequence the sequence $\displaystyle \left( \int_{\Omega}\Phi(\nabla u_n|)\,dx \right)$ has limit which will be denoted by $L$, that is, 
$$
\lim_{n \to +\infty} \int_{\Omega}\Phi(|\nabla u_n|)\,dx=L.
$$
As the functional $Q$ given in (\ref{Q}) is l.s.c. with respect to the weak$^*$ topology  we obtain
\begin{equation} \label{Z1}
\int_{\Omega}\Phi(|\nabla u|)\,dx \leq \lim_{n \to +\infty}\int_{\Omega}\Phi(|\nabla u_n|)\,dx=L.
\end{equation}
From (\ref{sequencia2}), we know that
$$
\int_{\Omega}\Phi(|\nabla v|)\,dx-\int_{\Omega}\Phi(|\nabla u_n|)\,dx \geq \int_{\Omega}f(u_n)(v-u_n)\,dx- \tau_n\|v-u_n\|, \quad \forall v \in W^{1,\Phi}_0(\Omega) \quad \mbox{and} \quad n \in\mathbb{N},
$$
from where it follows that
$$
\int_{\Omega}\Phi(|\nabla v|)\,dx-\int_{\Omega}\Phi(|\nabla u|)\,dx \geq \int_{\Omega}f(u)(v-u)\,dx, \quad \forall v \in W^{1,\Phi}_0(\Omega).
$$
From this, $u \in D(I)$ and it is a critical point of $I$. Moreover, we also have
$$
\int_{\Omega}\Phi(|\nabla u|)\,dx-\int_{\Omega}\Phi(|\nabla u_n|)\,dx \geq \int_{\Omega}f(u_n)(u-u_n)\,dx- \tau_n\|u-u_n\|, \quad  n \in\mathbb{N}.
$$
Therefore, 
\begin{equation} \label{Z2}
\int_{\Omega}\Phi(|\nabla u|)\,dx \geq \lim_{n \to +\infty} \int_{\Omega}\Phi(|\nabla u_n|)\,dx.
\end{equation}
Combining (\ref{Z1}) with (\ref{Z2}) we get  
\begin{equation} \label{PHI1}
\lim_{n \to +\infty} \int_{\Omega}\Phi(|\nabla u_n|)\,dx=\int_{\Omega}\Phi(|\nabla u|)\,dx. 
\end{equation}
From (\ref{FUNT1}) and (\ref{PHI1}),  
$$
\lim_{n \to +\infty}I(u_n)=I(u).
$$
In the sequel, we will show that $u \in \text{dom}(\phi(t)t)$. By Lemma \ref{dominio}, there is $(v_n) \subset \text{dom}(\phi(t)t)$ such that 
$$
\|v_n-u_n\| \leq 1/n \quad \mbox{and} \quad \int_{\Omega}\Phi(|\nabla v_n|)\,dx \leq \int_{\Omega}\Phi(|\nabla u_n|)\,dx, \quad \forall n \in \mathbb{N}.
$$
Consequently, 
$$
v_n \to u \quad \mbox{in} \quad C(\overline{\Omega})
$$
and
$$
\int_{\Omega}\Phi(|\nabla v|)\,dx-\int_{\Omega}\Phi(|\nabla v_n|)\,dx \geq \int_{\Omega}f(u_n)(v-u_n)\,dx- |\tau_n| \|v-u_n\|, \quad \forall v \in W^{1,\Phi}_0(\Omega) \quad \mbox{and} \quad \forall n \in\mathbb{N}.
$$
Setting $v=v_n-\frac{1}{n}v_n$, we get
$$
\int_{\Omega}\Phi(|\nabla v_n-\frac{1}{n}\nabla v_n|)\,dx-\int_{\Omega}\Phi(|\nabla v_n|)\,dx \geq \int_{\Omega}f(u_n)(v_n-\frac{1}{n}v_n-u_n)\,dx- |\tau_n| \|v_n-\frac{1}{n}v_n-u_n\|,
$$
or equivalently
$$
\int_{\Omega}\frac{(\Phi(|\nabla v_n-\frac{1}{n}\nabla v_n|)-\Phi(|\nabla v_n|))}{-\frac{1}{n}}\,dx \leq -n\int_{\Omega}f(u_n)(v_n-u_n)\,dx+ \int_{\Omega}f(u_n)v_n \,dx+ n|\tau_n|\|v_n-u_n\|+|\tau_n|\|v_n\|.
$$
As $(u_n)$ is bounded in $W_0^{1,\Phi}(\Omega)$, $(f(u_n))$ is bounded in $L^{\infty}(\Omega)$, $(\tau_n)$ is bounded in $\mathbb{R}$ and $\|v_n-u_n\| \leq \frac{1}{n}$, it follows that the right side of the above inequality is bounded. Therefore, there is $M>0$ such that 
$$
\int_{\Omega}\frac{(\Phi(|\nabla v_n-\frac{1}{n}\nabla v_n|)-\Phi(|\nabla v_n|))}{-\frac{1}{n}}\, dx  \leq M, \quad \forall n \in \mathbb{N}.
$$ 
Since $\Phi$ is $C^{1}$, there is $\theta_n(x) \in [0,1]$ verifying
$$
\frac{\Phi(|\nabla v_n-\frac{1}{n}\nabla v_n|)-\Phi(|\nabla v_n|)}{-\frac{1}{n}}=\phi(|(1-{\theta_n(x)}/{n})\nabla v_n|)(1-{\theta_n(x)}/{n})|\nabla v_n|^{2}.
$$
Recalling that $0< 1-{\theta_n(x)}/{n}\leq 1$, we know that
$$
1-{\theta_n(x)}/{n} \geq (1-{\theta_n(x)}/{n})^{2},
$$
which leads to
$$
\int_{\Omega}\phi(|(1-{\theta_n(x)}/{n})\nabla v_n|)(1-{\theta_n(x)}/{n})^{2}|\nabla v_n|^{2}\,dx \leq M \quad \forall n \in \mathbb{N}.
$$
As $u_n \stackrel{*}{\rightharpoonup} u$ in $W^{1,\Phi}_0(\Omega)$, we also have $(1-{\theta_n(x)}/{n})v_n \stackrel{*}{\rightharpoonup} u$ in $W^{1,\Phi}_0(\Omega).$ Then, by using the fact that $\phi(t)t^{2}$ is convex, we can apply \cite[Lemma 3.2]{GKMS} to get 
$$
\liminf_{n \to +\infty}\int_{\Omega}\phi(|(1-{\theta_n(x)}/{n})\nabla v_n|)(1-{\theta_n(x)}/{n})^{2}|\nabla v_n|^{2}\,dx  \geq \int_{\Omega}\phi(|\nabla u|)|\nabla u|^{2}\,dx 
$$
and so,
$$
\int_{\Omega}\phi(|\nabla u|)|\nabla u|^{2}\,dx \leq M. 
$$
Recalling that 
$$
\phi(t)t^{2}=\Phi(t)+\Phi^*(\phi(t)t), \quad \forall t \in \mathbb{R}
$$
we have 
$$
\phi(|\nabla u|)|\nabla u|^{2}=\Phi(|\nabla u|)+\Phi^*(\phi(|\nabla u|)|\nabla u|)
$$
which leads to 
$$
\int_{\Omega}\phi(|\nabla u|)|\nabla u|^{2}\,dx=\int_{\Omega}\Phi(|\nabla u|)\,dx+\int_{\Omega}\Phi^*(\phi(|\nabla u|)|\nabla u|)\,dx.
$$
Since $\displaystyle \int_{\Omega}\phi(|\nabla u|)|\nabla u|^{2}\,dx$ and $\displaystyle \int_{\Omega}\Phi(|\nabla u|)\,dx$ are finite, we see that $\displaystyle \int_{\Omega}\Phi^*(\phi(|\nabla u|)|\nabla u|)\,dx$ is also finite, showing that $u \in \text{dom}(\phi(t)t)$, finishing the proof. 
\end{proof}

As an immediate consequence of the last proposition we have 
\begin{corollary} \label{pontocritico} Let $u \in W_0^{1,\Phi}(\Omega)$ be a  critical point of $I$, that is,  $0 \in \partial I(u)$. Then, $u$ is a weak solution of $(P)$.
\end{corollary}
\begin{proof} It is enough to apply the Proposition \ref{Lema2}  with $u_n=u$ for all $n \in \mathbb{N}$.  
\end{proof}

\subsection{Proof of Theorem \ref{T1}}
\begin{proof} From Lemmas \ref{Lema1} and \ref{Lema2},  $I$ verifies the assumptions of the mountain pass theorem due to Szulkin \cite{Szulkin}. Then the mountain pass level $\beta$ of $I$ is a critical level, that is, there is $u \in W^{1,\Phi}_0(\Omega)$ such that
$$
I(u)=\beta>0 \quad \mbox{and} \quad \int_{\Omega}\Phi(|\nabla v|)\,dx - \int_{\Omega}\Phi(|\nabla u|)\,dx \geq \int_{\Omega}f(u)(v-u)\,dx, \quad \forall v \in W^{1,\Phi}_0(\Omega).
$$
Thus, by Corollary  \ref{pontocritico} $u$ is a nontrivial solution of $(P)$.
\end{proof}

\section{Global Minimization}

In this section, we intend to prove Theorem \ref{T2} by showing that $I$ has a critical point which can be obtained by global minimization.

\begin{proof} By using the definition of $I$ and $(f_3)$, we get
$$
I(u) \geq \int_{\Omega}\Phi(|\nabla u|)\,dx -b_1\int_{\Omega}(\Phi(u/d))^{s}\,dx, \quad \forall u \in W^{1,\Phi}_0(\Omega).
$$
By  H\"older's inequality and (\ref{PI}),
$$
I(u) \geq \int_{\Omega}\Phi(|\nabla u|)\,dx -C\left(\int_{\Omega}\Phi(|\nabla u|)\,dx \right)^{s}, \quad \forall u \in W^{1,\Phi}_0(\Omega).
$$
Now, as $s \in (0,1)$ and
$$
\|u\| \to +\infty \Rightarrow \int_{\Omega}\Phi(|\nabla u|)\,dx \to +\infty,
$$	
we derive
$$
I(u) \to +\infty \quad \mbox{as} \quad \|u\| \to +\infty,
$$
showing that $I$ is coercive. This fact combined with the definition of $I$ gives that $I$ is bounded from below in  $W^{1,\Phi}_0(\Omega)$. Thereby, there is $(u_n) \subset W^{1,\Phi}_0(\Omega)$ such that
$$
I(u_n) \to I_{\infty}=\inf_{u \in W^{1,\Phi}_0(\Omega)}I(u) \quad \mbox{as} \quad n \to +\infty.
$$
Consequently, by coercivity of $I$, $(u_n)$ is bounded in $W^{1,\Phi}_0(\Omega)$. Thus, by Lemma \ref{Estrela},  for some subsequence,
$$
u_n \stackrel{*}{\rightharpoonup} u \quad \mbox{in} \quad  W^{1,\Phi}_0(\Omega).
$$
Now, applying \cite[Lemma 3.2]{GKMS} and \cite{ET}, the functional $I$ is weak$^*$ lower semicontinuous, and so,
$$
\liminf_{n \to +\infty}I(u_n) \geq I(u),
$$
implying that
$$
I(u)=I_{\infty}.
$$
Therefore $u \in D(I)$ and $0 \in \partial I(u)$, from where it follows that $u$ is weak solution of $(P)$. Now, we will prove that $u \not= 0$. To this end, it is enough to show that $I_{\infty}<0$.  Fix $v \in  C^{\infty}_{0}(\Omega)$ with $v \not= 0$, and note that by $(f_4)$, if $t>0$ is small enough,
$$
F(tv(x)) \geq c_1\Phi(tv(x)), \quad \forall x \in \overline{\Omega}.
$$
Thereby,
$$
\begin{array}{ll}
I(tv) & \leq \displaystyle \int_{\Omega}\Phi(t|\nabla v|)\,dx-c_1\int_{\Omega}(\Phi(tv))^{\gamma}\,dx\\
\mbox{} & \mbox{} \\
&  \leq A_1\Phi(tA_2)-B_1(\Phi(tB_2))^{\gamma}.
\end{array}
$$
From $(\phi_8)$, we see that $I(tv)<0$ for $t$ small enough. As $I_\infty \leq I(tv)$ , it follows that $I_\infty <0$, finishing the proof.
\end{proof}	

\section{The concave and convex case}
In this section, our intention is showing the Theorem \ref{T3}. Before proving this result, we recall that in this section the energy functional $I:W^{1,\Phi}_0(\Omega) \to \mathbb{R}$ is given by
$$
I(u)=\int_{\Omega}\Phi(|\nabla u|)\, dx-\frac{\lambda}{\alpha}\int_{\Omega}(\Phi(|u|))^{\alpha}\, dx
-\frac{1}{q}\int_{\Omega}(\Phi(|u|))^{q}\, dx.
$$
\subsection{First solution}
\begin{proof}
By using H\"older and Poincar\'e inequalities,
$$
I(u) \geq \int_{\Omega}\Phi(|\nabla u|)\, dx -\frac{\lambda}{\alpha}\left(\int_{\Omega}\Phi(|\nabla u|)\, dx \right)^{\alpha}-\frac{1}{q}\left(\int_{\Omega}\Phi(|\nabla u|)\, dx\right)^{q}.
$$

From the above inequality, there are positive numbers $\lambda^*, r$ and $\rho>0$ such that
\begin{equation} \label{E4}
I(u)>\rho \quad \mbox{for} \quad \int_{\Omega}\Phi(|\nabla u|)\, dx=r, \quad  \mbox{and} \quad 0<\lambda \leq \lambda^*.
\end{equation}
Hereafter, we denote by $X \subset W^{1,\Phi}_0(\Omega)$ the following closed set
$$
X=\left\{u \in W^{1,\Phi}_0(\Omega)\,:\, \int_{\Omega}\Phi(|\nabla u|)\, dx \leq r \right\},
$$
and by $I_\infty \in [0, +\infty)$ the number
$$
I_\infty=\inf_{u \in X}I(u).
$$
Arguing as in Section 3, it is possible to ensure that there exists $w \in int(X)$ with $I(w)<0$. This information implies that
\begin{equation} \label{E5}
\inf_{u \in X}I(u) < \inf_{u \in \partial X}I(u).
\end{equation}

By Using the Ekeland's variational principle, we find a sequence $(u_n) \subset X$ verifying
\begin{equation} \label{E6}
I(u_n) \to I_\infty \quad \mbox{and} \quad I(v)-I(u_n) \geq -\frac{1}{n}\|v-u_n\| \quad \forall v \in X \setminus \{u_n\}.
\end{equation}
Since the functionals $J$ is Gateaux differentiable at $u_n$ and $Q$ is convex, we derive that there exists $\tau_n \to 0$ verifying 
$$
\int_{\Omega}\Phi(|\nabla v|)\,dx-\int_{\Omega}\Phi(|\nabla u_n|)\,dx \geq \int_{\Omega}f(u_n)(v-u_n)\,dx- \tau_n\|v-u_n\|, \quad \forall v \in W^{1,\Phi}_0(\Omega) \quad \mbox{and} \quad n \in\mathbb{N}.
$$
The above analysis gives that $(u_n)$ is a $(PS)$ sequence for $I$.

A simple computation shows that $(f_5)$ leads to
$$
\lim_{t \to +\infty}\frac{F(t)}{h(t)tf(t)}=\frac{1}{q}<1,
$$
from where it follows that condition  $(f_2)$ is verified. Thereby, arguing as in Proposition \ref{Lema2} of Section 3, functional $I$ verifies the $(PS)$ condition, and thus, there is $u \in X$ such that
$$
I(u)= I_\infty<0 \quad \mbox{and} \quad  0 \in \partial I(u).
$$
Therefore, $u$ is our first nontrivial weak solution.
\end{proof}

\subsection{Second solution}
\begin{proof}  By above arguments, we know that $f$ satisfies $(f_2)$ and $(\ref{E4})$ guarantees $I$  verifies the mountain pass geometry. Thereby, the same arguments explored in Section 3 work to show that $I$ possesses a  critical point $w \in W^{1,\Phi}_0(\Omega)$ at the mountain pass level $\beta$ of $I$, that is,
$$
I(w)=\beta>0 \quad \mbox{and} \quad 0 \in\partial I(w).
$$
Thus, $w$ is a nontrivial solution. Moreover, $w$ is not equal to first solution $u$, because $I(u)<0<I(w)$. Therefore, $w$ is our second nontrivial weak solution.
	
\end{proof}

\end{document}